\documentclass[10pt]{amsart}

\usepackage{amssymb,amsmath,amsthm,amscd}

\advance\oddsidemargin by -1cm
\advance\evensidemargin by -1cm
\newtheorem{theorem}{Theorem}

\newtheorem{proposition}{Proposition}

\newtheorem{lemma}{Lemma}

\theoremstyle{definition}

\newcommand{\bdm}{\begin{displaymath}}
\newcommand{\edm}{\end{displaymath}}
\newcommand{\bq}{\begin{equation}}
\newcommand{\eq}{\end{equation}}
\newcommand{\bqn}{\begin{equation*}}
\newcommand{\eqn}{\end{equation*}}

\newcommand{\rn}{\mathbb{R}^n}

\newcommand{\norm}[1]{\left\| #1 \right\|}
\newcommand{\mklm}[1]{\left\{ #1 \right\}}

\renewcommand{\d}{\,d}
\newcommand{\N}{{\mathbb N}}

\newcommand{\C}{{\mathbb C}}
\newcommand{\R}{{\mathbb R}}

\newcommand{\D}{{\mathcal D}}
\newcommand{\E}{{\mathcal E}}
\newcommand{\F}{{\mathcal F}}

\newcommand{\T}{{\mathcal T}}

\newcommand{\1}{{\bf 1}}
\renewcommand{\epsilon}{\varepsilon}
\renewcommand{\phi}{\varphi}
\renewcommand{\rho}{\varrho}

\newcommand{\Cinft}{{\rm C^{\infty}}}
\newcommand{\CT}{{\rm C^{\infty}_c}}

\renewcommand{\L}{{\rm L}}
\newcommand{\Lcal}{{\mathcal L}}

\newcommand{\Sym}{{\rm S}}

\newcommand{\GL}{\mathrm{GL}}

\newcommand{\g}{{\bf \mathfrak g}}

\renewcommand{\det}{\mathrm{det}\,}
\renewcommand{\Im}{\mathrm{Im}\,}

\newcommand{\vol}{\mathrm{vol}\,}

\newcommand{\Crit}{\mathrm{Crit}}

\DeclareMathOperator{\supp}{supp}

\DeclareMathOperator{\tr}{tr}
\DeclareMathOperator{\gd}{\partial}

\newcommand{\e}[1]{\,{\mathrm e}^{#1}\,}
\newcommand{\dbar}{{\,\raisebox{-.1ex}{\={}}\!\!\!\!d}}

\begin{document} 

%%% Cuanto trabajo!!!!!!! Que el Espiritu Santo te ilumine!

\author{Pablo Ramacher}
\title{Equivariant Lefschetz formulae and heat asymptotics} 
\address{Pablo Ramacher, Philipps-Universit\"at Marburg, Fachbereich Mathematik und Informatik, Hans-Meer\-wein-Str., 35032 Marburg, Germany}
\email{ramacher@mathematik.uni-marburg.de}
\thanks{This work was supported by the DFG grant RA 1370/2-1}.
\date{August 10, 2011}

\begin{abstract}
We prove an equivariant Lefschetz formula for elliptic complexes over a compact manifold carrying the action of a compact Lie group of isometries  via  heat equation methods. 
\end{abstract}

\maketitle

\setcounter{tocdepth}{1}
\tableofcontents

\section{Introduction}

  The computation of the Lefschetz number of an endomorphism of an elliptic complex constitutes a generalization of the index problem for an elliptic operator. For geometric endomorphisms arising from transversal mappings, this computation was accomplished by Atiyah and Bott  in \cite{atiyah-bott67}, generalizing the classical Lefschetz fixed point theorem.  In this paper, we shall prove a local formula for the equivariant Lefschetz number of an elliptic complex over a compact manifold carrying the action of a compact Lie group of isometries.

To explain our result, let $M$  be a compact Riemannian manifold of dimension $n$, and $G$ a compact Lie group acting effectively and isometrically on $M$. Consider further a family  $E_0, \dots, E_N$ of $\Cinft$-vector bundles over $M$, and let
\bqn
\Cinft(E_0)\stackrel{P_0}{\longrightarrow} \Cinft(E_1) \stackrel{P_1}{\longrightarrow} \dots \stackrel{P_{N-1}}{\longrightarrow} \Cinft(E_N)
\eqn
be an elliptic complex $\E$ on $M$.  Assume that for every $g \in G$ and $0 \leq j \leq N$ there exist smooth bundle homomorphisms $\Phi_j(g): g^\ast E_j \rightarrow E_j$, and  define the linear maps 
\bqn 
T_j(g):\Cinft(E_j)\stackrel{g^\ast}{\longrightarrow} \Cinft(g^\ast E_j)\stackrel{\Phi_j(g)}{\longrightarrow} \Cinft(E_j), \quad T_j(g) s(x) = \Phi_j(g) [  s(gx)].
\eqn
If $P_j T_j(g)=T_{j+1}(g) P_j$ for each $g \in G$ and $0 \leq j \leq N-1$, the maps $T_j(g)$ constitute a geometric endomorphism $T(g)$ of the complex $\E$, and we denote the corresponding endomorphisms on the cohomology groups $H^j(\E,\C)$  by  $\T_j(g)$. Since the cohomology groups are finite--dimensional, one can define the {Lefschetz number}  for each of the $T(g)$ by 
\bqn 
L(T(g))=\sum_{j=0}^N (-1)^{j} \tr {\T_j(g)}_{|H^j(\E,\C)}, \qquad g \in G.
\eqn
In case  that $g:M \rightarrow M$ has only simple fixed points,  the Lefschetz fixed point theorem of Atiyah and Bott  expresses $L(T(g))$ as a sum over  fixed points of $g$.   Consider now a unitary irreducible representation $(\pi_\rho,V_\rho)$ of $G$ associated to a character $\rho \in \hat G$, and write $T=\mklm{T(g)}_{g \in G}$.  We  then define the  $\rho$-equivariant Lefschetz number of $T$ as
\bqn 
\Lcal_\rho(T)= \frac 1 {\vol G} \int_G L(T(g)) \overline{\rho(g)} \d_G(g),
\eqn
where $d_G$ denotes a Haar measure on $G$. 
Note that if $G$ is trivial, this simply reduces to the Euler--Poincar\'e characteristic of $\E$. Our aim is to  prove a local  formula for $\Lcal_\rho(T)$,  based on  asymptotics of the heat equation. For this we shall approximate the heat operator by pseudodifferential operators. This leads to the problem of determining the asymptotic behavior of certain oscillatory integrals which have been examined before in \cite{ramacher10} during the study  of the spectrum of an invariant elliptic operator. The existence of such local formulae  for  $\Lcal_\rho(T)$ suggests that,  using invariance theory, it should be possible to find global expressions for  $\Lcal_\rho(T)$  in terms of characteristic classes. This will be the subject of   a subsequent paper. 

The original proof of Atiyah and Bott  of the Lefschetz fixed point theorem relies on the theory of pseudodifferential operators, and an extension of the trace of a finite rank operator to a larger class of maps including geometric endomorphisms. They also gave an alternative proof   based on  work of Seeley \cite{seeley67} on the zeta--function of an elliptic operator,  going back to work of Minakshisundaram and Pleijel \cite{minakshisundaram-pleijel}. It was first pointed out by H\"{o}rmander that heat equation techniques can be used instead of fractional powers to obtain a local formula for the Lefschetz number of a geometric endomorphism.  Following this approach,  Kotake gave another proof of the Atiyah--Bott fixed point theorem in \cite{kotake69}. 
Based on work of  McKean and Singer \cite{mckean-singer67}, Patodi \cite{patodi71} and Gilkey \cite{gilkey73}, this development finally culminated in a proof of the index theorem by Atiyah, Bott and Patodi using heat equation methods \cite{atiyah-bott-patodi73}. These methods were then applied to derive generalized Lefschetz fixed point formulae for the classical complexes. In the case of the signature complex,  Donnelly-Patodi \cite{donnelly-patodi77} and Kawasaki \cite{kawasaki78} gave a new proof of the $G$-signature theorem of Atiyah--Singer, while the other  complexes were treated in Gilkey \cite{gilkey79}.

 The paper is structured as follows. In Section \ref{sec:AB} we review the Lefschetz formula of Atiyah and Bott for elliptic complexes, and state the main result of this paper. Section \ref{sec:HELN} introduces the heat equation, and explains how it is related to the index problem. The crucial observation here, which is due to Bott,  is that the Lefschetz number of an elliptic complex can be expressed as an alternating sum of heat traces. In Section \ref{sec:SEHA} the heat operator is approximated by pseudodifferential operators, obtaining an expansion for the equivariant heat trace and for $\Lcal_\rho(T)$ in terms of oscillatory integrals. Their asymptotic behavior is described in Section \ref{sec:SEARS} using the stationary phase theorem and resolution of singularities. A local formula for $\Lcal_\rho(T)$ in then derived in Section \ref{sec:LF}, while an outlook is given in Section \ref{sec:O}.

\section{Equivariant Lefschetz formulae for elliptic complexes}
\label{sec:AB}

We begin by reviewing the classical Atiyah--Bott fixed point formula for elliptic complexes following \cite{atiyah-bott67}. Let $M$ be a compact $\Cinft$-manifold of dimension $n$, and $E$ and $F$ complex vector bundles over $M$. Denote the corresponding spaces of smooth sections by $\Cinft(E)$ and $\Cinft(F)$, respectively, and consider a differential operator
\bqn 
P: \Cinft(E) \longrightarrow \Cinft (F)
\eqn
of order $d$ between $E$ and $F$, which is a linear map given locally by a matrix of partial differential operators with smooth coefficients. Let $T^\ast M$ be the cotangent bundle of $M$, and $\pi: T^\ast M \rightarrow M$  the canonical projection. The terms of order $d$ of $D$ define in an invariant manner a bundle homomorphism
\bqn 
p_d:\pi^\ast E \longrightarrow \pi^\ast F
\eqn
over the cotangent space $T^\ast M$ called the \emph{principal symbol} of $D$.  If $p_d$ is an isomorphism away from the zero section of $T^\ast M$, the operator $D$ is called \emph{elliptic}. Let now $E_0, \dots, E_N$ be a family of $\Cinft$-vector bundles over $M$. Then a sequence 
\bq
\label{eq:a}
\Cinft(E_0)\stackrel{P_0}{\longrightarrow} \Cinft(E_1) \stackrel{P_1}{\longrightarrow} \dots \stackrel{P_{N-1}}{\longrightarrow} \Cinft(E_N)
\eq
of differential operators is called an \emph{elliptic complex} if $P_j P_{j-1}=0$ for all $1 \leq j \leq N-1$, and the sequence of corresponding principal symbols
\bqn 
0 \longrightarrow  \pi^\ast E_0 \stackrel{p_{0,d_0}}{\longrightarrow}\pi^\ast E_1 \stackrel{p_{1,d_1}}{\longrightarrow}  \dots \stackrel{p_{N-1,d_{N-1}}}{\longrightarrow} \pi^\ast E_N \longrightarrow 0
\eqn
is exact outside the zero section. The complex \eqref{eq:a} is denoted by $\E$, and its cohomology groups are  defined as usual according to 
\bqn 
H^j(\E,\C) = \ker P_j/\Im P_{j-1}.
\eqn
As one can show, these cohomology groups are all finite--dimensional for an elliptic complex. In particular, the Euler--Poincar\'{e} characteristic
\bqn 
\chi(\E)= \sum_{j=0}^N (-1)^j \dim H^j(\E,\C) 
\eqn
is well defined.

Consider next an \emph{endomorphism} $T$ of the  elliptic complex \eqref{eq:a}, by  which one means a sequence of linear maps $T_j:\Cinft(E_j) \rightarrow \Cinft(E_j)$ such that $P_jT_j =T_{j+1} P_j$. Since both $\ker P_j$ and $\Im P_{j-1}$ are left invariant by $T_j$, such an endomorphism  induces endomorphisms $\T_j$ on the cohomology groups $H^j(\E,\C)$. Since the latter are finite--dimensional, one can define the \emph{Lefschetz number}  of $T$ as
\bqn 
L(T)=\sum_{j=0}^N (-1)^{j} \tr {\T_j}_{|H^j(\E,\C)}.
\eqn
Clearly, if $T$ is the identity, $L(T)$ just reduces to the Euler-Poincar\'e characteristic $\chi(\E)$ of the complex $\E$. In  case that $N=1$, $\chi(\E)$ is just the index of the elliptic operator $P_0$. Therefore, the  computation of the Lefschetz number $L(T)$ constitutes  a generalization of the index problem for an elliptic operator, which was solved by Atiyah--Singer in \cite{atiyah-singer63}.

Let now  $g:M \rightarrow M$ be a smooth map, so that for each $0\leq j\leq N$ we have the induced bundles $g^\ast E_j$ over $M$, together with the linear maps 
\bqn 
g^\ast:\Cinft(E_j)\longrightarrow \Cinft(g^\ast E_j), \quad (g^\ast s)(x) = s(gx).
\eqn
In addition, assume that we are given smooth bundle homomorphisms $\Phi_j(g): g^\ast E_j \rightarrow E_j$. We can then define the linear maps 
\bq
\label{eq:b}
T_j(g):\Cinft(E_j)\stackrel{g^\ast}{\longrightarrow} \Cinft(g^\ast E_j)\stackrel{\Phi_j(g)}{\longrightarrow} \Cinft(E_j), \quad T_j s(x) = \Phi_j(g) [  s(gx)].
\eq
If  $P_j T_j(g)=T_{j+1}(g) P_j$, the system consisting of $g$ and the linear maps $T_j(g):\Cinft(E_j) \rightarrow \Cinft(E_j)$ is called a \emph{geometric endomorphism} of $\E$. If $g$ has only \emph{simple} fixed points, meaning that $\det (\1 - dg_x)\not=0$ for each fixed point $x\in M$, the mapping $g$ is called \emph{transversal}. In this case, each fixed point is isolated so that, $M$ being compact, the set of fixed points $\mathrm{Fix}(g)$ of $g$ is finite.  Note that at a fixed point $x \in \mathrm{Fix} (g)$, $\Phi_{j}(g)_x$ is an endomorphism of the fiber $E_{j,x}$, so that its trace $\tr \Phi_{j}(g)_x$ is defined. After these preparations, we can state
\begin{theorem}[Atiyah--Bott--Lefschetz fixed point theorem] 
\label{thm:AB}
Consider a geometric endomorphism $T(g)$ of an elliptic complex \eqref{eq:a}, given by a transversal mapping $g:M\rightarrow M$, and bundle homomorphisms $\Phi_j(g):g^\ast E_j \rightarrow E_j$. 
Then the Lefschetz number $L(T(g))$ of $T(g)$ is given by 
\bqn
L(T(g))=\sum_{x \in \mathrm{Fix}(g)} \sum_{j=0}^N \frac{(-1)^{j} \tr \Phi_{j}(g)_x}{|\det (\1 -dg_x)|}.
\eqn
\end{theorem}
\begin{proof}
See Atiyah--Bott \cite{atiyah-bott67}.
\end{proof}

 The classical example for an elliptic complex is the De--Rham complex. In this case, the $j^{th}$ exterior powers of $dg$  yield a geometric endomorphism, and the above theorem reduces to the classical Lefschetz fixed point formula.

Consider now a compact Lie group $G$, acting effectively and isometrically on $M$. Let us assume that for every $g \in G$ and $0 \leq j \leq N$ there exist smooth bundle homomorphisms $\Phi_j(g): g^\ast E_j \rightarrow E_j$, so that we can define the linear maps 
\eqref{eq:b}. In addition, we shall  assume that  $P_j T_j(g)=T_{j+1}(g) P_j$ for each $g \in G$ and $0 \leq j \leq N-1$. Under these conditions, the mappings $T_j(g)$ define  geometric endomorphisms  $T(g)$ of $\E$ for each $g \in G$,  and we write $T=\mklm{T(g)}_{g \in G}$. 
The  Lefschetz number  of $T(g)$ is  given by
\bq
\label{eq:c}
L(T(g))=\sum_{j=0}^N (-1)^{j} \tr {\T_j(g)}_{|H^j(\E,\C)},
\eq
where the $\T_j(g)$ denote the endomorphisms induced by the  maps $T_j(g)$ on the cohomology groups $H^j(\E,\C)$. In what follows, we shall consider the following generalization of the Euler-Poincar\'e characteristic of $\E$. Let $(\pi_\rho,V_\rho)$ be a unitary irreducible representation  of $G$ associated to the character $\rho \in \hat G$. We then define the \emph{$\rho$-equivariant Lefschetz number} of $T$ as
\bq
\label{eq:d}
\Lcal_\rho(T)= \frac 1 {\vol G} \int_G L(T(g)) \overline{\rho(g)} \d_G(g),
\eq
where $d_G$ is a Haar measure on $G$. Clearly, if $G$ is trivial, this just reduces to $\chi(\E)$. The main result of this paper is the following local formula for $\Lcal_\rho(T)$.

\begin{theorem}
\label{thm:main}
Let $M$ be a compact Riemannian manifold of dimension $n$, and $G$ a compact Lie group acting effectively and isometrically on $M$. Let $\mathbb{J}: T^\ast M\rightarrow \g^\ast$ be the momentum map of the induced  Hamiltonian action on the cotangent bundle $T^\ast M$, and put $\Xi=\mathbb{J}^{-1}(0)$. Consider further  an elliptic complex $\E$  on $M$, together with  a family of geometric endomorphisms $T=\mklm{T(g)}_{g \in G}$ of $\E$ defined by the isometries $g: M \rightarrow M$, and bundle homomorphisms $\Phi_j(g):g^\ast E_j\rightarrow E_j$, and denote by  $\Delta_j$  the associated Laplacians.  Let $\mklm{(\kappa_\gamma,U^\gamma)}$ be an atlas of $M$, $\mklm{f_\gamma}$ a subordinated partition of unity, and $\{\phi^\gamma_{E_{j}}\}$  corresponding trivializations of the bundles $E_j$. 
\begin{enumerate}
\item For each $\rho \in \hat G$,  the $\rho$-equivariant Lefschetz number $\Lcal_\rho(T)$ of $T$ is given by the local formula
\begin{align*}
\Lcal_\rho(T)&= \frac { (2\pi)^{\kappa-n}} {\vol G} \sum_{j=0}^N (-1)^j    \sum_\gamma  \Big [ \mathcal{L}_{j,n-\kappa,\gamma} + \mathcal{R}_{j,\gamma} \Big ],
\end{align*}
where $\kappa$ is the dimension of a principal $G$-orbit in $M$, and 
\bqn
\mathcal{L}_{j,k,\gamma}=\int_{\mathrm{Reg}\, \mathcal{C}} \frac { f_\gamma(x)   \cdot \tr \Big [ \Phi_j(g)_{x} \circ (\phi_{E_j}^\gamma)^{-1}_{x} \circ e_k^\gamma(1,\kappa_\gamma(x),\eta,\Delta_j)  \circ (\phi_{E_j}^\gamma)_x \Big ] \cdot \overline{\rho(g)} }{|\det   \, \Phi_\gamma''(x, \eta,g)_{N_{(x, \eta, g)}\mathrm{Reg}\, \mathcal{C}}|^{1/2}} \d(\mathrm{Reg}\, \mathcal{C})(x,\eta,g).
\eqn
 $\mathrm{Reg}\, \mathcal{C}$ denotes the regular part of  the critical set $\mathcal{C}=\mklm{(x,\xi,g) \in \Xi \times G: g \cdot (x,\xi) =(x,\xi)}$ of the phase functions $\Phi_\gamma(x,\eta,g)=(\kappa_\gamma(gx) - \kappa_\gamma(x))\cdot \eta$, and $\d(\mathrm{Reg}\, \mathcal{C})$ the induced volume density.  The $ e_{k}^\gamma(1,\kappa_\gamma(x),\eta,\Delta_j)$,  where  $0 \leq k \leq n-\kappa$, are local symbols,  and  the  remainder terms $\mathcal{R}_{j,\gamma}$ are given in terms of   local symbols up to order  $n-\kappa-1$. 
\item If the endomorphisms $\Phi_j(g)_x$ act trivially on the fibers $E_{j,x}$, 
\begin{gather*}
\Lcal_\rho(T)= \frac { (2\pi)^{\kappa-n}[{\pi_\rho}_{|H}:1]} {\vol G} \sum_{j=0}^N    \sum_\gamma (-1)^j \\
\cdot \Big [ \int_{{\mathrm{Reg}} \, \Xi}  f_\gamma(x)   \cdot \tr \Big [  (\phi_{E_j}^\gamma)^{-1}_{x} \circ e_{n-\kappa}^\gamma(1,\kappa_\gamma(x),\eta,\Delta_j)  \circ (\phi_{E_j}^\gamma)_x \Big ]  \frac{d({\mathrm{Reg}}\, \Xi)(x, \eta)}{\vol\mathcal{O}_{(x,\eta)}} + \mathcal{R}_{j,\gamma} \Big ],
 \end{gather*}
where   $H\subset G$ a principal isotropy group,  and $ [{\pi_\rho}_{|H}:1]$  the multiplicity of the trivial representation in the restriction of $\pi_\rho$ to $H$, while $\mathcal{O}_{(x,\eta)}$ denotes the $G$-orbit in $T^\ast M$ through $(x,\eta)$. 
\end{enumerate}
\end{theorem}

In contrast to the Atiyah--Bott fixed point theorem, higher dimensional fixed point sets are now involved. The proof of Theorem \ref{thm:main} will therefore require the more elaborate techniques of fractional powers and the heat equation, which were not needed in the original proof of Theorem \ref{thm:AB}. 

\section{Heat equation and Lefschetz numbers}
\label{sec:HELN}

Let $M$ be a closed $n$-dimensional Riemannian manifold, $dM$ its volume density, and $E$ a complex $\Cinft$- vector bundle over $M$ endowed with a smooth Hermitian metric $h$. Under these assumptions,  $\Cinft(E)$  becomes a Pre--Hilbert space with  inner product
\bqn 
(s,s')_{\L^2}= \int_M h(s(x),s'(x)) \d M,\qquad s, s' \in \Cinft(E).
\eqn
 Its completion is given by  the Hilbert space $\L^2(E)$ of square integrable sections of $E$. Denote by $\Omega$ the density bundle on $M$, which is the line bundle associated to the tangent bundle $TM$ via the representation $A \to |\det A|$ of $\GL(n,\R)$. Consider further $E^\ast$, the dual bundle of $E$, and set $E'=E^\ast \otimes \Omega$. Let $\Cinft(E')'$ be the dual topological vector space of $\Cinft(E')$. An element of  $\D'(E)=\Cinft(E')'$ is  called a \emph{distributional  section} of $E$. In general, if
\begin{equation*}
  A: \Cinft(E) \longrightarrow \Cinft(F) 
\end{equation*}
is a continuous linear operator, its Schwartz kernel $K_A$ is a distributional section on $M\times M$ of the bundle $F\boxtimes E'$. Here $F\boxtimes E'$ denotes the exterior tensor product of $F$ and $E'$, which is the smooth bundle over $M \times M$ with fibers $F_x\otimes E'_y$, $x, y \in M$.  Suppose now  that 
\bq
\label{eq:e}
P:\Cinft(E) \longrightarrow \L^2(E)
\eq
is  an elliptic differential operator of order $m$ on $E$, regarded as an operator in $\L^2(E)$ with domain $\Cinft(E)$, and assume that $P$ is symmetric and positive\footnote{The positivity of $P$ means that, outside the zero section of $T^\ast M$, the  principal symbol is given by a positive definite matrix.}. Then  $P$ has discrete spectrum, and there exists an orthonormal basis of $\L^2(E)$ consisting of smooth sections $\mklm{e_j}$ such that $Pe_j=\lambda_j e_j$, $|\lambda_j| \to \infty$.  Associated to $P$, we consider the heat equation
\bqn 
(\gd_t +P ) h(x,t)=0, \qquad \lim_{t \to 0} h(x,t) =f(x), \qquad t >0,
\eqn
with initial condition $f \in \Cinft(E)$. It is a parabolic differential equation, and its solution is given by $h(x,t) =e^{-tP} f(x)$, where 
\bq
\label{eq:f}
e^{-tP} = \frac 1 {2\pi i} \int_\Gamma e^{-t\lambda} (P-\lambda \1)^{-1} \d \lambda
\eq
is the corresponding heat operator. Here $\Gamma$ is a suitable path in $\C$ coming from infinity and going to infinity such that $(P-\lambda \1)$ is invertible for $\lambda \in \Gamma$. The heat operator has a smooth kernel $K_{e^{-tP}}\in \Cinft(E\boxtimes E')$, which for each $x,y \in M$ defines an element $ K(t,x,y,P)dM(y)  \in \mathrm{Hom} (E_y,E_x) \otimes \Omega_y$. As Seeley showed in \cite{seeley67}, $e^{-tP}$ is of $\L^2$-trace class,  its trace being given by
\bqn 
\tr_{\L^2} (e^{-tP}) =\sum_j (e^{-tP} e_j,e_j)_{\L^2}=\sum_j e^{-t\lambda_j}= \int _M \tr K(t,x,x,P) \d M(x).
\eqn
Let next $\E$ be an elliptic complex over $M$ as in \eqref{eq:a}, where each of the bundles $E_j$ is equipped with a smooth Hermitian metric.  For simplicity, we shall assume that all the $P_j$ have the same order. Consider the adjoint complex 
\bqn
\Cinft(E_0)\stackrel{P_0^\ast}{\longleftarrow} \Cinft(E_1) \stackrel{P_1^\ast }{\longleftarrow} \dots \stackrel{P_{N-1}^\ast}{\longleftarrow} \Cinft(E_N),
\eqn
where the $P_j^\ast$ are differential operators determined uniquely by the condition
$(P_j s, s')_{\L^2} = (s, P_j^\ast s')_{\L^2}$ for all $ s \in \Cinft(E_j)$, $s'\in \Cinft(E_{j+1})$, and  define the associated Laplacians
\bqn 
\Delta_j= P_{j-1} P_{j-1}^\ast + P_j^\ast P_j.
\eqn
Then $\Delta_j:\Cinft(E) \rightarrow \L^2(E)$ is an elliptic, symmetric and  positive operator. 

Suppose now that a compact Lie group $G$ acts effectively and isometrically on $M$, and  that for every $g \in G$ and $0 \leq j \leq N$ there exist smooth bundle homomorphisms $\Phi_j(g): g^\ast E_j \rightarrow E_j$, so that we can define the linear maps 
\eqref{eq:b}. Assume that  $P_j T_j(g)=T_{j+1}(g) P_j$ for each $g \in G$ and $0 \leq j \leq N-1$, and let $L(T(g))$ be the Lefschetz number  of the geometric endomorphism $T(g)$ defined in \eqref{eq:c}.   The following algebraic observation is due to Bott, and is a direct consequence of the Hodge decomposition theorem. It is crucial for the heat equation approach to the index problem. 
\begin{lemma}
\label{lemma:1}
Let $\E$ be an elliptic complex, and $\Delta_j$ the associated Laplacians. For each $g \in G$, let $T(g)$ be  a gemetric endomorphism  determined by the action of $g$ on $M$, and smooth bundle homomorphisms $\Phi_j(g)$. Then 
\bqn 
L(T(g))=\sum_{j=0}^N (-1)^j  \tr_{\L^2} \big (T_j(g) e^{-t\Delta_j}\big )
\eqn
for any $t>0$.
\end{lemma}
\begin{proof}
See Atiyah--Bott \cite{atiyah-bott67},  Section 8, Kotake \cite{kotake69}, Lemma 3, or Gilkey,  \cite{gilkey95}  Lemma 1.10.1.
\end{proof}

Next, let $(\pi_\rho,V_\rho)$ be a unitary irreducible representation  of $G$ associated to the character $\rho \in \hat G$. 
In what follows, we shall use Lemma \ref{lemma:1} to prove a local formula for the $\rho$-equivariant Lefschetz number $L_\rho(T)$ introduced in \eqref{eq:d}. For this, we shall require  an asymptotic expansion for 
\bqn 
\int_G \tr_{\L^2} \big (T_j(g) e^{-t\Delta_j}\big ) \overline{\rho(g)} \d_G(g), \qquad t \to 0^+,
\eqn
which will be derived in the next sections.

\section{Pseudodifferential operators and equivariant heat asymptotics}
\label{sec:SEHA}

 Our aim is to give a local formula for the $\rho$-equivariant Lefschetz number $\Lcal_\rho(T)$  using the alternating sum formula of  Lemma \ref{lemma:1}, and asymptotics of the heat equation. For this, we shall first  construct an approximation of the heat operator by pseudodifferential operators. Let $\tilde U$ be an open set in $\R^n$. Recall that a continuous linear operator
\begin{displaymath}
  A:\CT(\tilde U) \longrightarrow \Cinft(\tilde U)
\end{displaymath}
 is called a \emph{pseudodifferential operator} if it can be written in the form
\begin{equation}
  \label{I}
Au(\tilde x)=\int e^{i\tilde x \cdot \xi} a(\tilde x,\xi) \hat u(\xi) \dbar \xi,
\end{equation}
where $\hat u=\mathcal{F}(u)$ denotes the Fourier transform of $u$,  $\dbar \xi=(2\pi)^{-n} \d \xi$, and $a(\tilde x,\xi)\in \Cinft(\tilde U \times \rn)$ is an amplitude with the following property.  There is an $l \in \R$ such that for any multiindices $\alpha,\beta$, and any compact set $K\subset \tilde U$, there exist constants $C_{\alpha,\beta,K}$ for which
\begin{equation*}
  \label{H}
|\gd ^\alpha_\xi\gd ^\beta_{\tilde x} a(\tilde x,\xi)| \leq C_{\alpha,\beta,K} (1+|\xi|^2)^{(l-|\alpha|)/2}, \qquad \tilde x \in K,  \quad \xi \in \R^n,
\end{equation*}
where  $|\alpha|=\alpha_1+\dots+\alpha_n$. The class of all such functions $a(\tilde x,\xi)$ is denoted by  $\Sym^l(\tilde U\times \R^n)$, and the class of operators of the form \eqref{I}  with $a(\tilde x,\xi) \in \Sym^l(\tilde U\times \R^n)$, by $\L^l(\tilde U)$.
 In particular, one  puts $\Sym^{-\infty}(\tilde U\times \R^n)=\bigcap _{l \in \R} \Sym^l(\tilde U\times \R^n)$. 

 Consider next an $n$-dimensional $\Cinft$--manifold $M$, and let  $\mklm{(\kappa_\gamma, U^\gamma)}$ be an atlas for $M$. 
Write $\tilde U^\gamma= \kappa_\gamma(U^\gamma)\subset \rn$. If $\pi_E:E\rightarrow M$ and $\pi_F:F\rightarrow M$ are smooth vector bundles over $M$ trivialized by
\begin{equation*}
\label{IIIa}
  \alpha_{E}^\gamma :E_{|U^\gamma}\longrightarrow U^\gamma \times \C^{e}, \mathrm{e}\mapsto (\pi_E(\mathrm{e}), \phi_E^\gamma(\mathrm{e})), \qquad  \alpha_{F}^\gamma :F_{|U^\gamma}\longrightarrow U^\gamma \times \C^{f},  \mathrm{f}\mapsto (\pi_F(\mathrm{f}), \phi_F^\gamma(\mathrm{f})),
\end{equation*}
then a continuous linear operator 
\begin{equation*}
  A:\CT(E)\longrightarrow \Cinft(F)
\end{equation*}
is called a \emph{pseudodifferential operator between sections of $E$ and $F$ of order $l$}, if for any $U^\gamma$ there is a $f\times e $-matrix of  pseudodifferential operators $\tilde A_{ij} \in \L^l(\tilde U^\gamma)$ such that 
\begin{equation*}
\label{IIIb}
  (\phi_F^\gamma\circ (Av)_{|U^\gamma})_i=\sum_{j} A_{ij} (\phi_E^\gamma\circ v)_j, \qquad v \in \CT(U^\gamma; E),
\end{equation*}
where the $A_{ij}$ are defined by the relations $A_{ij} u=[\tilde A_{ij} (u \circ \kappa_\gamma^{-1} ) ] \circ \kappa_\gamma$, $u \in \CT(U^\gamma)$. 
In this case we write $A \in \L^l(M;E,F)$, or simply $\L^l(E,F)$.  As explained before, the  Schwartz kernel $K_A$ of $A$ is a distribution section on $M\times M$ of the bundle $F\boxtimes E'$. For an introduction into the theory of pseudodifferential operators, the reader is referred to \cite{shubin} or \cite{hoermanderIII}. 

Suppose now that $M$ is a closed Riemannian manifold, and $E$ a complex smooth vector bundle over $M$ with a smooth Hermitian metric, and let $P:\Cinft(E) \rightarrow \L^2(E)$ be an elliptic differential operator as in \eqref{eq:e}. Consider the heat operator  associated to $P$, and let $\Gamma\subset \C$ be the path specified in \eqref{eq:f}.  $P$ is locally given by a matrix of differential operators $P_{ij}^\gamma \in \L^l(\tilde U^\gamma)$ with symbols   $p_{ij}^\gamma(\tilde x,\xi) \in S^l(\tilde U^\gamma \times \rn)$. On each chart $U^\gamma$, the symbol of $P$ is  represented  by the matrix $p^\gamma(\tilde x,\xi)=(p_{ij}^\gamma(\tilde x,\xi))_{ij}$, and we decompose the latter into its homogeneous components 
\bqn 
p^\gamma(\tilde x,\xi)=p_m^\gamma(\tilde x,\xi)+ \dots +p_0^\gamma(\tilde x,\xi),
\eqn
where $m$ is the order of $P$. 
The positivity of $P$ means that $p_m^\gamma(\tilde x,\xi)$ is a positive definite matrix for $\xi \not=0$, and together with the ellipticity and the symmetry  of $P$ this implies that $p_m^\gamma(\tilde x,\xi) - \lambda$ is invertible  for $\lambda \in \Gamma$. Write $(\kappa_\gamma^{-1})^\ast \d M= \beta_\gamma \d \tilde y$. We now recursively define the local symbols 
\begin{align*}
r_0^\gamma(\tilde x,\xi,\lambda,P) &=(p_m^\gamma(\tilde x,\xi) - \lambda)^{-1}, \\
r_k^\gamma(\tilde x,\xi,\lambda, P)&=-r_0^\gamma(\tilde x,\xi,\lambda,P) \left ( \sum\limits_{|\beta|+m+l'-l=k, l'<k} \frac {(-i)^{|\beta|}} {\beta!}\,  (\gd^\beta_\xi p_{l}^\gamma)(\tilde x,\xi) \cdot (\gd^\beta_{\tilde x}r_{l'}^\gamma)(\tilde x,\xi,\lambda,P) \right),
\end{align*}
as well as 
\bqn 
e_k^\gamma(t,\tilde x,\xi,P)= \frac 1 {2\pi i} \int_\Gamma e^{-t\lambda} r_k^\gamma(\tilde x,\xi,\lambda,P) \d \lambda,\qquad t >0,
\eqn
 and consider the corresponding pseudodifferential operators
\begin{align*}
[\tilde R^\gamma_k(\lambda,P) v](\tilde x)&= \int e^{i \tilde x\cdot \eta} r_k^\gamma(\tilde x,\eta,\lambda,P) \widehat{ v \beta_\gamma}(\eta) \dbar \eta, \\
[\tilde E^\gamma_k(t,P) v](\tilde x)&= \int e^{i \tilde x\cdot \eta} e_k^\gamma(t,\tilde x,\eta,P) \widehat{ v \beta_\gamma}(\eta) \dbar \eta, 
\end{align*}
where $v \in \CT(\tilde U^\gamma; \C^e)$. With these definitions,  set 
\begin{align*}
 R^\gamma_k(\lambda,P)u &=(\phi_E^\gamma)^{-1} \circ \big [\tilde R^\gamma_k(\lambda,P)\big (  \phi_E^\gamma \circ u \circ \kappa_\gamma^{-1}\big ) \big ]\circ \kappa_\gamma, \\  E^\gamma_k(t,P) u &=(\phi_E^\gamma)^{-1}\circ \big [\tilde E^\gamma_k(t,P)\big (  \phi_E^\gamma \circ u \circ \kappa_\gamma^{-1}\big )\big ]  \circ \kappa_\gamma,
\end{align*}
where $u \in \CT(U^\gamma;E)$. 
Let $\mklm{f_\gamma}$ be a partition of unity subordinated to the atlas $\mklm{(\kappa_\gamma,U^\gamma)}$, and $\bar f_\gamma \in \CT(U^\gamma)$ test functions satisfying $\bar f_\gamma \equiv 1$ on $\supp f_\gamma$. Denote by $F_\gamma$ and $\bar F_\gamma$ the multiplication operators corresponding to $f_\gamma$ and $\bar f_\gamma$, respectively, and define on $M$ 
\begin{align}
\label{eq:8}
 R_K(\lambda,P)=&\sum_{k=0}^K \sum _\gamma F_\gamma \,    R_k^\gamma(\lambda,P)\,  \bar F_\gamma, \qquad 
 E_K(t,P)=\sum_{k=0}^K\sum _\gamma  F_\gamma \,  E_k^\gamma(t,P) \,  \bar F_\gamma.
\end{align}
Explicitly, one computes for $u \in \CT(E)$
\begin{align}
\begin{split}
\label{eq:9}
  F_\gamma \,  E_k^\gamma &(t,P) \,   \bar F_\gamma u(x)=f_\gamma(x) \Big [ (\phi_E^\gamma)^{-1} \circ \big [\tilde E_k^\gamma(t,P)(\phi_E^\gamma \circ \bar f_\gamma u \circ \kappa_\gamma^{-1}) \big ] \circ \kappa_\gamma \Big ] (x)\\
  &=f_\gamma(x) (\phi_E^\gamma)^{-1}  \left [ \int e^{i\kappa_\gamma(x) \cdot \eta} e^\gamma_k(t,\kappa_\gamma(x),\eta,P) \F\big ((\phi_E^\gamma \circ \bar f_\gamma u \circ \kappa_\gamma^{-1})\beta_\gamma\big )(\eta) \dbar \eta\right ]\\
  &=f_\gamma(x) (\phi_E^\gamma)^{-1}  \left [ \int_{\tilde U^\gamma} \int e^{i[\kappa_\gamma(x)-\tilde y] \cdot \eta} e^\gamma_k(t,\kappa_\gamma(x),\eta,P) (\phi_E^\gamma \circ \bar f_\gamma u )(\kappa_\gamma^{-1}(\tilde y)) \beta_\gamma(\tilde y) \dbar \eta \, d\tilde y\right ]\\
&=f_\gamma(x) (\phi_E^\gamma)^{-1}  \left [ \int_{U^\gamma} \int e^{i[\kappa_\gamma(x)-\kappa_\gamma( y)] \cdot \eta}  \bar f_\gamma(y) e^\gamma_k(t,\kappa_\gamma(x),\eta,P)   (\phi_E^\gamma \circ  u )(y)  \dbar \eta \, dM(y)\right ].
\end{split}
\end{align}

The operators $R_K(\lambda,P)$ and $E_K(t,P)$ are approximations of the resolvent $(P-\lambda \1)^{-1}$ and the heat operator $e^{-tP}$, respectively, as $K \to \infty$. More precisely, if $Q: H^s(E) \rightarrow H^{s'}(E)$ is an operator between Sobolev spaces of sections, define the operator norms
\bqn 
\norm{Q}_{s,s'}=\sup_{u \in \Cinft, u\not=0} \norm{Qu}_{s'} \norm{u}_{s}^{-1}. 
\eqn
Then, one has the following
\begin{lemma}\hspace{0cm}
\begin{enumerate}
\item For every $K\in \N$, we have 
$$(P-\lambda \1) R_K(\lambda,P) -\1\sim_K 0, \qquad R_K(\lambda,P)(P-\lambda \1) -\1 \sim_K 0.$$ 
\item The operators $E_K(t,P)$ have smooth kernels and for every $l\in \N$ there exists a $K(l)\in \N$ such that for $0<t<1$
\bqn 
\norm{e^{-tP} -E_K(t,P)}_{-l,l} \leq C_l t^l
\eqn
for all $K \geq K(l)$. 
\end{enumerate}
\end{lemma}
\begin{proof}
See Gilkey, \cite{gilkey95}, Lemmata 1.7.2 and 1.8.1.
\end{proof}
Let now $T_j(g)$ and $\Delta_j$ be as in Lemma \ref{lemma:1}, and $l \in \N$. Assertion (2) of the preceding lemma  implies that $E_K(t,P)$ is of $\L^2$-trace class, and for every $K\geq K(l)$
\bq 
\label{eq:10}
\tr_{\L^2} \big (T_j(g) e^{-t\Delta_j}\big )= \tr_{\L^2} \big (T_j(g) E_K(t,\Delta_j) \big )+O(t^l).
\eq
Since $T_j(g)= \Phi_j(g) \circ g^\ast$, \eqref{eq:8} and \eqref{eq:9} imply that 
\begin{align}
\begin{split}
\label{eq:11}
\tr_{\L^2} \big (T_j(g) &E_K(t,\Delta_j) \big )=\sum_{k=0}^K \sum_\gamma \int_M \int  e^{i[\kappa_\gamma(gx)-\kappa_\gamma( x)] \cdot \eta}  \bar f_\gamma(x) f_\gamma(gx)  \\ 
& \cdot \tr \Big [ \Phi_j(g)_{gx} \circ (\phi_{E_j}^\gamma)^{-1}_{gx} \circ e_k^\gamma(t,\kappa_\gamma(gx),\eta,\Delta_j)  \circ (\phi_{E_j}^\gamma)_x \Big  ] \dbar \eta \, dM(x)\\
&=\sum_{k=0}^K t^{\frac {k-n} m} \sum_\gamma \int_M \int  e^{\frac i {{t}^{1/m}}[\kappa_\gamma(gx)-\kappa_\gamma( x)] \cdot \eta}  \bar f_\gamma(x) f_\gamma(gx)  \\ 
& \cdot \tr \Big [ \Phi_j(g)_{gx} \circ (\phi_{E_j}^\gamma)^{-1}_{gx} \circ e_k^\gamma(1,\kappa_\gamma(gx),\eta,\Delta_j)  \circ (\phi_{E_j}^\gamma)_x \Big  ] \dbar \eta \, dM(x),
\end{split}
\end{align}
where we took into account that $e_k^\gamma(t,\kappa_\gamma(gx),\eta,\Delta_j)=t^{k/m} e_k^\gamma(1,\kappa_\gamma(gx), {t}^{1/m}\cdot  \eta ,\Delta_j)$. As a consequence of \eqref{eq:10}, we obtain the following

\begin{proposition}
\label{prop:1}
Let  $g \in G$ be fixed and $t>0$.  For every $l \in \N$ there exists a $K(l)\in \N$ such that 
\begin{align*}
\begin{split}
\tr_{\L^2} \big (T_j(g) e^{-t\Delta_j}\big )&=   \sum_{k=0}^K t^{\frac {k-n} m} \sum_\gamma \int_M \int  e^{\frac i {{t}^{1/m}}[\kappa_\gamma(gx)-\kappa_\gamma( x)] \cdot \eta}  \bar f_\gamma(x) f_\gamma(gx)  \\ 
& \cdot \tr \Big [ \Phi_j(g)_{gx} \circ (\phi_{E_j}^\gamma)^{-1}_{gx} \circ e_k^\gamma(1,\kappa_\gamma(gx),\eta,\Delta_j)  \circ (\phi_{E_j}^\gamma)_x \Big  ] \dbar \eta \, dM(x)+O(t^l)
\end{split}
\end{align*}
for all $K \geq K(l)$.
 \end{proposition} 
\qed

\noindent
Now, Lemma \ref{lemma:1} implies that for any $t>0$
\bqn 
\Lcal_\rho(T)=  \frac 1 {\vol G} \sum_{j=0}^N (-1)^j  \int_G  \tr_{\L^2} \big (T_j(g) e^{-t\Delta_j}\big ) \overline{\rho(g)} \d_G(g).
\eqn
With  Proposition \ref{prop:1} we therefore obtain
\begin{theorem}
\label{thm:3}
Let $\rho \in \hat G$. For every $l \in \N$ there exists a $K(l)\in \N$ such that 
\begin{align}
\begin{split}
\label{eq:local}
\Lcal_\rho(T)&= \frac 1 {\vol G} \sum_{j=0}^N (-1)^j  \int_G  \sum_{k=0}^K t^{\frac {k-n} m} \sum_\gamma \int_M \int  e^{\frac i {{t}^{1/m}}[\kappa_\gamma(gx)-\kappa_\gamma( x)] \cdot \eta}  \bar f_\gamma(x) f_\gamma(gx)  \\ 
& \cdot \tr \Big [ \Phi_j(g)_{gx} \circ (\phi_{E_j}^\gamma)^{-1}_{gx} \circ e_k^\gamma(1,\kappa_\gamma(gx),\eta,\Delta_j)  \circ (\phi_{E_j}^\gamma)_x \Big  ] \dbar \eta \, dM(x)\overline{\rho(g)} \d_G(g)+O(t^l)
\end{split}
\end{align}
for all $K \geq K(l)$ and any $t>0$. 
\end{theorem}
\qed

\noindent
The left-hand side of \eqref{eq:local} does not depend on $t>0$. In order to find a local formula for $\Lcal_\rho(T)$, we have to  find  an asymptotic expansion of the right-hand side, and determine the constant term. We are therefore left with the task of examining the asymptotic behavior of integrals of the form 
\begin{align}
\label{eq:int}
I(\mu)
&= \int _{T^\ast U}  \int_{G} e^{i  \Phi(x , \xi,g)/\mu }   a( g  x,  x , \xi,g)d_G(g) \d(T^\ast U)(x,\xi),  \qquad \mu \to 0^+,  
\end{align}
via the generalized stationary phase theorem, where $(\kappa,U)$ are local coordinates on $M$,  $\d(T^\ast U)(x,\xi)$  is the canonical volume density on $T^\ast U$, and $d_G(g)$ is the  volume density of a left invariant metric on $G$, while $a \in \CT(U \times T^\ast  U\times G)$ is an amplitude which does not depend on $\mu$, and
\bq
\label{eq:phase}
\Phi(x, \xi, g) =(\kappa(gx) - \kappa ( x)) \cdot  \xi.
\eq 
This will be done in the next section.

\section{Singular equivariant asymptotics and resolution of singularities}
\label{sec:SEARS}

To examine the asymptotic behavior of the integrals \eqref{eq:int} by means of the  stationary phase principle, we have to study the critical set of the phase function \eqref{eq:phase}.  Consider for this the cotangent bundle  $\pi:T^\ast M\rightarrow M$, as well as the tangent bundle $\tau: T(T^\ast M)\rightarrow T^\ast M$, and define on $T^\ast M$ the Liouville form 
\bqn 
\Theta(\mathfrak{X})=\tau(\mathfrak{X})[\pi_\ast(\mathfrak{X})], \qquad \mathfrak{X} \in T(T^\ast M).
\eqn
Regard $T^\ast M$ as a  symplectic manifold with symplectic form 
$
\omega= d\Theta
$, 
and define for any  element $X$ in the Lie algebra $\g$ of $G$ the function
\bqn
J_X: T^\ast M \longrightarrow \R, \quad \eta \mapsto \Theta(\widetilde{X})(\eta),
\eqn
where $\widetilde X$ denotes the fundamental vector field on $T^\ast M$, respectively $M$,  generated by $X$.   $G$ acts on $T^\ast M$ in a Hamiltonian way, and the corresponding symplectic momentum  map is  given by 
\bqn
\mathbb{J}:T^\ast M\to \g^\ast,  \quad \mathbb{J}(\eta)(X)=J_X(\eta).
\eqn
 Let $(\kappa,U)$ be local coordinates on $M$ as in \eqref{eq:int}, and write  $\kappa(x)=(\tilde x_1,\dots, \tilde x_n)$,  $\eta=\sum \xi_i (d\tilde x_i)_x \in T_x^\ast U$. One  computes then for any $X \in \g$
\begin{align*}
\frac d{dt} \Phi( x, \xi, \e{-tX})_{|t=0}&=\frac d{dt} \Big ( \kappa (\e{-tX}  x)  \cdot \xi \Big )_{|t=0}= \sum \xi_i \widetilde X_{x}(\tilde x_i)  = \sum \xi_i (d\tilde x_i)_{x}(\widetilde X_{x})\\ 
&=\eta (\widetilde X_{x})=\Theta(\widetilde X)(\eta) = \mathbb{J}(\eta)(X).
\end{align*}
Therefore $\Phi$ represents the global analogue of the momentum map. Further, one has
\bqn
\gd _{\tilde x} \Phi ( \kappa ^{-1} ( \tilde x), \xi, g) = [ \, ^T ( \kappa \circ g \circ \kappa ^{-1} )_{\ast,  \tilde x} - \1 ]  \xi= ( g^\ast_{\tilde x}-\1) \cdot \xi, 
\eqn
so that $\gd _x \Phi ( x, \xi, g )=0$ amounts precisely to the condition $g^\ast  \xi=\xi$. Since  $\gd_\xi \Phi(x,\xi,g)=0$ if, and only if $gx=x$, 
one obtains  
 \begin{align}
 \label{eq:16}
 \begin{split}
\mathcal{C}&= \Crit(\Phi)
=\mklm{(x,\xi,g) \in T^\ast U \times G: (\Phi_{\ast})_{(x,\xi,g)}=0} \\
&= \mklm{( x  ,\xi, g) \in (\Xi  \cap T^\ast U)\times G:  \,  g\cdot (x,\xi)=(x,\xi)},
\end{split}
\end{align}
  where $
  \Xi=\mathbb{J}^{-1}(0)
$
is the zero level of the momentum  map. If  $G$ acts on $M$ only with one orbit type, the critical set of the phase function $\Phi(x,\xi,g)$ is clean. In this case, the stationary phase method can directly be applied  to  yield an asymptotic expansion of the integrals $I(\mu)$.

\begin{proposition}
\label{prop:asympexp}
Let $M$ be a connected, closed Riemannian manifold, and $G$  a compact, connected Lie group $G$ of isometries acting on $M$ with one orbit type. 
Consider further an oscillatory integral  $I(\mu) $ of the form  \eqref{eq:int}. We then have the asymptotic expansion
\bqn 
I(\mu)\sim(2\pi \mu)^{ \kappa}\sum_{j=0} ^{ \infty} \mu^j Q_j (\Phi;a)
\eqn
as $\mu \to 0^+$, 
where $\kappa$ denotes the dimension of an orbit of principal type, and  the coefficients $Q_j(\Phi;a)$ can be computed explicitly. In particular, one has
\bqn
Q_0(\Phi;a)= \int _{\mathcal{C}} \frac {a(m)}{|\det \Phi''(m)_{|N_m\mathcal{C}}|^{1/2}} d\sigma_{\mathcal{C}}(m),
\eqn
where $d\sigma_{\mathcal{C}}$ is the induced volume density on the critical set $\mathcal{C}=\Crit(\Phi)$ given by \eqref{eq:16}.
\end{proposition}
\begin{proof}
It is not hard to see that under the  assumption that $G$ acts on $M$ only with one orbit type, $\Phi(x,\xi,g)$ has a clean critical set, meaning that 
\begin{enumerate}
\item[(I)] $\mathcal{C}$ is a smooth submanifold of $M$  of codimension $2\kappa$;
\item[(II)] at each point $x \in \mathcal{C}$, the Hessian $\Phi''(x)$  of $\Phi$ is transversally non-degenerate, i.e. non-degenerate on $T _xM/ T_xC \simeq N_x\mathcal{C}$, where $N_m\mathcal{C}$ denotes the normal space to $\mathcal{C}$ at $x$.
\end{enumerate}
\noindent
The generalized stationary phase theorem \cite{ramacher10}, Theorem 5, then implies that for all $N \in \N$, there exists a constant $C_{N,\Phi,a}>0$ such that
\bqn
\Big |I(\mu) - e^{i\Phi_0/\mu}(2\pi \mu)^{ {\kappa}}\sum_{j=0} ^{N-1} \mu^j Q_j (\Phi;a)\Big | \leq C_{N,\Phi,a} \mu^N,
\eqn
where  $\Phi_0$ is the constant value of $\Phi$ on $\mathcal{C}$. Furthermore, the $Q_j(\Phi;a)$ can be computed explicitly, and  for each $j$ there exists a constant $\tilde C_{j,\Phi,a}>0$ such that 
\bqn
|Q_j(\Phi;a)|\leq \tilde C_{j,\Phi,a}.
\eqn
In particular,
\bqn
Q_0(\Phi;a)= \int _{\mathcal{C}} \frac {a(m)}{|\det \Phi''(m)_{|N_m\mathcal{C}}|^{1/2}} d\sigma_{\mathcal{C}}(m) e^{ i \frac\pi 4 \sigma_{\Phi''}},
\eqn
where $d\sigma_{\mathcal{C}}$ is the induced volume density on $\mathcal{C}$, and 
 $\sigma_{\Phi''}$  the constant value of the signature of the transversal Hessian $\Phi''(m)_{|N_m\mathcal{C}}$ on $\mathcal{C}$. Since $\Phi_0=0$, one  computes for arbitrary $ N \in \N$
 \begin{align*} 
\Big  |I(\mu) &- (2\pi \mu)^{ \kappa}\sum_{j=0} ^{ N-1} \mu^j Q_j (\Phi;a)\Big |\leq \Big  |I(\mu) - (2\pi \mu)^{ \kappa}\sum_{j=0} ^{ \kappa +N-1} \mu^j Q_j (\Phi;a)\Big |+\Big  | (2\pi \mu)^{ \kappa}\sum_{j=N} ^{\kappa+  N-1} \mu^j Q_j (\Phi;a)\Big |\\
&\leq   C_{\kappa +N,\Phi,a} \mu^{\kappa+N} + (2 \pi \mu)^\kappa \sum_{j=N}^{\kappa + N -1} \mu^j \tilde C_{j,\Phi,a}=O(\mu^{\kappa+N}), 
 \end{align*}
yielding the proposition.
\end{proof}

In general, the major difficulty resides in the fact that, unless the $G$-action on $T^\ast M$ is free, the considered momentum map is not a submersion, so that $\Xi$ and  $\mathcal{C}=\Crit(\Phi)$ are not  smooth manifolds. The stationary phase theorem can therefore not immediately be applied to the integrals  $I(\mu)$. Nevertheless, it was shown in \cite{ramacher10} that by resolving the singularities of the critical set $\mathcal{C}$, and applying the stationary phase theorem in a suitable resolution space, an asymptotic description of $I(\mu)$ can be obtained. More precisely, one has the following
\begin{theorem} 
\label{thm:I(mu)}
Let $M$ be a connected, closed Riemannian manifold, and $G$  a compact, connected Lie group $G$ acting isometrically and effectively on $M$. Consider the oscillatory integral
\begin{align*}
I(\mu)
&= \int _{T^\ast U}  \int_{G} e^{i  \Phi(x , \xi,g)/\mu }   a( g  x,  x , \xi,g)d_G(g) \d(T^\ast U)(  x,  \xi),  \qquad \mu \to 0^+,  
\end{align*}
where $(\kappa,U)$ are local coordinates on $M$,  $ \d(T^\ast U)(  x,  \xi)$ is the canonical volume density on $T^\ast U$, and $d_G(g)$ the volume density  on $G$ with respect to some left invariant metric on $G$, while $a \in \CT(U \times T^\ast  U\times G)$ is an amplitude, and
$
\Phi(x, \xi, g) =(\kappa(g x) - \kappa ( x)) \cdot  \xi
$.  
Then $I(\mu)$ has the asymptotic expansion 
\bqn 
I(\mu) = (2\pi \mu)^{\kappa} \mathcal{L}_0 + O\big (\mu^{\kappa+1}(\log \mu^{-1})^{\Lambda-1}\big ),  \qquad \mu \to 0^+.
\eqn
Here $\kappa$ is the dimension of an orbit of principal type in $M$, $\Lambda$ the maximal number of elements of a totally ordered subset of the set of isotropy types, and the leading coefficient is given by 
\bq
\label{eq:L0}
\mathcal{L}_0=\int_{\mathrm{Reg}\, \mathcal{C}} \frac { a( g  x,  x , \xi,g) }{|\det   \, \Phi''(x, \xi,g)_{N_{(x, \xi, g)}\mathrm{Reg}\, \mathcal{C}}|^{1/2}} \d(\mathrm{Reg}\, \mathcal{C})(x,\xi,g),
\eq
where $\mathrm{Reg}\, \mathcal{C}$ denotes the regular part of  $\mathcal{C}=\mklm{(x,\xi,g) \in \Xi \times G: g \cdot (x,\xi) =(x,\xi)}$, and $\d(\mathrm{Reg}\, \mathcal{C})$ the induced volume density. 
In particular, the integral over $\mathrm{Reg}\, \mathcal{C}$ exists.
\end{theorem}
\begin{proof}
See \cite{ramacher10}, Theorem 11. 
\end{proof}

\section{A local formula for $\mathcal{L}_\rho(T)$}
\label{sec:LF}

We are now able to derive a local formula for $\mathcal{L}_\rho(T)$. Let us begin with the non-singular case.

\begin{proposition}
Let $M$ be a connected, closed Riemannian manifold, and $G$  a compact, connected Lie group $G$ of isometries acting on $M$ with one orbit type. Take  $\rho\in \hat G$, and let   $\Lcal_\rho(T)$ be the $\rho$-equivariant Lefschetz number defined in \eqref{eq:d}. Then
\begin{align*}
\begin{split}
\Lcal_\rho(T)=\frac {(2\pi)^{\kappa-n}} {\vol G} \sum_{j=0}^N \sum_{k+q=n-\kappa} \sum_\gamma   (-1)^j Q_q(\Phi;a_{j,k,\gamma} ),
 \end{split}
\end{align*}
where the coefficients $Q_q(\Phi;a_{j,k,\gamma} )$ can be computed explicitly.
\end{proposition}
\begin{proof}
As an immediate consequence of Theorem \ref{thm:3} and Proposition \ref{prop:asympexp}, for any $l$ and $L \in \N$ there exists a $K(l)$ such that  for all $K \geq K(l)$ one has the expansion
\begin{align*}
\begin{split}
\Lcal_\rho(T)&= \frac {1} {\vol G} \,  t^{\frac{\kappa-n}{m}}\sum_{j=0}^N (-1)^j   \sum_{k=0}^K t^{\frac {k} m} \sum_\gamma  \Big [ (2\pi)^{\kappa-n} \sum_{q=0}^{L-1} t^{\frac qm} Q_q(\Phi;a_{j,k,\gamma}) + O(t^{\frac Lm}) \Big ] +O(t^l)
\end{split}
\end{align*}
for any $t>0$, where 
\bq
\label{eq:18}
a_{j,k,\gamma}(x,\eta,g)= \bar f_\gamma(x) f_\gamma(gx)   \cdot \tr \Big [ \Phi_j(g)_{gx} \circ (\phi_{E_j}^\gamma)^{-1}_{gx} \circ e_k^\gamma(1,\kappa_\gamma(gx),\eta,\Delta_j)  \circ (\phi_{E_j}^\gamma)_x \Big ] \cdot \overline{\rho(g)}.
\eq
Choose  $l,L > n-\kappa$. Since $\mathcal{L}_\rho(T)$ is independent of $t$, it is equal to the constant term in this expansion, while all other  terms  must vanish. The assertion now follows.
\end{proof}

We come now to the general case and to the proof of the  main result. 

\begin{proof}[Proof of Theorem \ref{thm:main}]
By Theorems \ref{thm:3} and \ref{thm:I(mu)}, for any $l$ there exists a $K(l)$ such that 
\begin{align}
\label{eq:19}
\begin{split}
\Lcal_\rho(T)&= \frac {1} {\vol G} \,  t^{\frac{\kappa-n}{m}}\sum_{j=0}^N (-1)^j   \sum_{k=0}^K t^{\frac {k} m} \sum_\gamma  \Big [ (2\pi)^{\kappa-n} \mathcal{L}_{j,k,\gamma} + O(t^{\frac 1m}(\log t^{-1})^{\Lambda-1}) \Big ]
\end{split}
\end{align}
up to terms of order $O(t^l)$ for all $K \geq K(l)$ and any $t>0$, where 
\bqn
\mathcal{L}_{j,k,\gamma}=\int_{\mathrm{Reg}\, \mathcal{C}} \frac {a_{j,k,\gamma}(x,\eta,g) }{|\det   \, \Phi_\gamma''(x, \eta,g)_{N_{(x, \eta, g)}\mathrm{Reg}\, \mathcal{C}}|^{1/2}} \d(\mathrm{Reg}\, \mathcal{C})(x,\eta,g),
\eqn
$\Phi_\gamma(x,\eta,g)=(\kappa_\gamma(gx) - \kappa_\gamma(x))\cdot \eta$, and $a_{j,k,\gamma}(x,\eta,g)$ restricted to $\mathcal{C}$ is given by
\bqn 
a_{j,k,\gamma}(x,\eta,g)=  f_\gamma(x)   \cdot \tr \Big [ \Phi_j(g)_{x} \circ (\phi_{E_j}^\gamma)^{-1}_{x} \circ e_k^\gamma(1,\kappa_\gamma(x),\eta,\Delta_j)  \circ (\phi_{E_j}^\gamma)_x \Big ] \cdot \overline{\rho(g)}.
\eqn
Note that at a fixed point $x$, $\Phi_j(g)_x$ is an endomorphism of $E_{j,x}$, so that the above trace is well defined. Choose $l > n-\kappa$. Since $\mathcal{L}_\rho(T)$ must be equal to the constant term in the expansion \eqref{eq:19}, one finally obtains  the equality
\begin{align*}
\Lcal_\rho(T)&= \frac { (2\pi)^{\kappa-n}} {\vol G} \sum_{j=0}^N (-1)^j    \sum_\gamma  \Big [ \mathcal{L}_{j,n-\kappa,\gamma} + \mathcal{R}_{j,\gamma} \Big ],
\end{align*}
where the remainder terms $\mathcal{R}_{j,\gamma}$ do depend on  amplitudes $a_{j,k,\gamma}$ with $0 \leq k \leq n-\kappa-1$. This proves Assertion (1) of Theorem \ref{thm:main}.  Assume now that $\Phi_j(g)_x$ acts trivially on $E_{j,x}$, and recall that  for any  smooth, compactly supported function $\alpha$ on $\Xi  \cap T^\ast U^\gamma$ one has the formula
\bqn
\int_{{\mathrm{Reg}} \, {\mathcal{C}}}\frac{\overline {\rho(g)}  \alpha(x,\eta)}{|\det  \, \Phi_\gamma'' (x,\eta,g)_{|N_{(x, \eta,g)}{\mathrm{Reg}} \, {\mathcal{C}}_\gamma} |^{1/2}} d({\mathrm{Reg}}  \, {\mathcal{C}})(x, \eta,g)
=[{\pi_\rho}_{|H}:1]\int_{{\mathrm{Reg}} \, \Xi} \alpha (x, \eta) \frac{d({\mathrm{Reg}}\, \Xi)(x, \eta)}{\mbox{vol }\mathcal{O}_{(x,\eta)}},
\eqn
where $H$ is a principal isotropy group, and $[ \pi_{\rho|H}: 1]$ denotes the multiplicity of the trivial representation in the restriction of $\pi_\rho$ to $H$,  while $\mathcal{O}_{(x,\eta)}$ is the orbit in $T^\ast M$ through $(x,\eta)$, compare \cite{cassanas-ramacher09}, Lemma 7. In this case,  
\bqn 
\mathcal{L}_{j,k,\gamma}=[{\pi_\rho}_{|H}:1]\int_{{\mathrm{Reg}} \, \Xi}  f_\gamma(x)   \cdot \tr \Big [  (\phi_{E_j}^\gamma)^{-1}_{x} \circ e_k^\gamma(1,\kappa_\gamma(x),\eta,\Delta_j)  \circ (\phi_{E_j}^\gamma)_x \Big ]  \frac{d({\mathrm{Reg}}\, \Xi)(x, \eta)}{\mbox{vol }\mathcal{O}_{(x,\eta)}},
\eqn
and we obtain Assertion (2) of Theorem \ref{thm:main}.
\end{proof}

\section{Outlook}
\label{sec:O}

A few years after the index theorem was proved by heat equation methods, the same techniques were employed to derive generalized Lefschetz fixed point formulae. Thus, for fixed $g \in G$, asymptotic expansions for $\tr_{\L^2} \big (T_j(g) e^{-t\Delta_j}\big )$ were obtained by Gilkey and Lee, see \cite{gilkey95}, Lemma 1.10.2, and also by Donnelly \cite{donnelly76}.  As they showed,
\bq
\label{eq:20}
\tr_{\L^2} \big (T_j(g) e^{-t\Delta_j}\big )\sim \sum_i \sum_k  t^{\frac{k-n_{g,i}}m} \int_{N_{g,i}} a_k(x,\Delta_j, T_j(g)) dN_{g,i}(x),
\eq
where the $N_{g,i}$ are the connected components  of dimension $n_{g,i}$ of the fixed point set of $g:M \rightarrow M$, and the $a_k(x,\Delta_j,T_j(g))$  are scalar invariants depending functorially on the symbol of $\Delta_j$ and on $T_j(g)$. The existence of such expansions strongly suggested new proofs of the Atiyah--Singer--Lefschetz fixed point formulae for compact group actions. 
 In the case of isolated fixed points, Kotake gave an expansion of $\tr_{\L^2} \big (T_j(g) e^{-t\Delta_j}\big )$ in \cite{kotake69}, which was sufficient to  give a new proof of Theorem \ref{thm:AB} by heat equation methods. In general, as a consequence of  the expansion \eqref{eq:20} and  Lemma \ref{lemma:1}, one has the local formulae
\bqn 
{L}(T(g))=\sum_i \sum_{j=0}^N (-1)^j \int_{N_{g,i}} a_{n_i}(x,\Delta_j,T_j(g)) dN_{g,i}(x),
\eqn
and using invariance theory, the terms in these formulae can be identified as characteristic classes. In this way, Donnelly-Patodi \cite{donnelly-patodi77} and Kawasaki \cite{kawasaki78} gave a new proof of the $G$-signature theorem of Atiyah--Singer in the case of the signature complex, while the other classical complexes were treated in Gilkey \cite{gilkey79}. In the same way, Theorem \ref{thm:main} suggests that it should be possible to find global expressions for  $\Lcal_\rho(T)$  in terms of characteristic classes  using invariance theory. This will be pursued in a subsequent paper, and should lead to  topological formulae relating characteristic classes of fixed point sets on $M$ to characteristic classes of the symplectic quotient $\Xi/G$.

%%%%%% Bibliography %%%%%%%%%%%%%%%%%%%%%%%%%%%%%

\providecommand{\bysame}{\leavevmode\hbox to3em{\hrulefill}\thinspace}
\providecommand{\MR}{\relax\ifhmode\unskip\space\fi MR }
\providecommand{\MRhref}[2]{%
  \href{http://www.ams.org/mathscinet-getitem?mr=#1}{#2}
}
\providecommand{\href}[2]{#2}

%\bibliography{bibliography}
%\bibliographystyle{amsplain}

\end{document}